\documentclass[a4paper]{amsart}
\usepackage{amsmath}
\usepackage{amsfonts}

\newtheorem{theorem}{Theorem}
\theoremstyle{plain}

\newtheorem{definition}{Definition}

\newtheorem{lemma}{Lemma}

\input{tcilatex}

\begin{document}
\title[Inverse CFSL problems]{Inverse problems for a conformable fractional
Sturm-Liouville operator}
\author{A. Sinan Ozkan}
\curraddr{Department of Mathematics, Faculty of Science, Sivas Cumhuriyet
University 58140 Sivas, Turkey}
\email{sozkan@cumhuriyet.edu.tr}
\author{\.{I}brahim Adalar}
\curraddr{Zara Veysel Dursun Colleges of Applied Sciences, Sivas Cumhuriyet
University Zara/Sivas, Turkey}
\email{iadalar@cumhuriyet.edu.tr}

\begin{abstract}
In this paper, a Sturm-Liouville boundary value problem equiped with
conformable fractional derivates is considered. We give some uniqueness
theorems for the solutions of inverse problems according to the Weyl
function, two given spectra and classical spectral data. We also study on
half-inverse problem and prove a Hochstadt and Lieberman-type theorem.
\end{abstract}

\maketitle

\section{\textbf{Introduction }}

Inverse spectral problems consist in recovering the coefficients of an
operator from some given data; for example Weyl function, spectral function,
nodal points and some special sequences which consist of some spectral
values. Various inverse problems for the classical Sturm-Liouville operator
have been studied for about ninety years (see \cite{Ambar}, \cite{Borg}-\cite%
{horv}, \cite{lev}, \cite{levitan}, \cite{marc}, \cite{marc2}, \cite{troos}, 
\cite{ozk} and the references therein). Since these kinds of problems appear
in mathematical physics, mechanics, electronics, geophysics and other
branches of natural sciences the literature on this area is vast.

Fractional derivative which is as old as calculus appears by a question of
L'Hospital to Leibniz in 1695. He asked what does it mean $\frac{d^{n}f}{%
dx^{n}}$ if $n=1/2$. Later on, many researchers tried to give a definition
of a fractional derivative. Most of them used an integral form for the
fractional derivative (see \cite{mil}, \cite{old}). However, almost all of
them fail to satisfy some of the basic properties owned by usual
derivatives, for example chain rule, the product rule, mean value theorem
and etc. In 2014, the authors Khalil et al. introduced a new simple
well-behaved definition of the fractional derivative called conformable
fractional derivative\ \cite{hal}. One year later, Abdeljawad gave the
fractional versions of some important concepts e.g. chain rule, exponential
functions, Gronwall's inequality, integration by parts, Taylor power series
expansions and etc \cite{abd}. Also, other basic properties on conformable
derivative can be found in \cite{atan}. It seems to satisfy all the
requirements of the standard derivative. Because of its effectiveness and
applicability, conformable derivative has received a lot of attention and
has applied quickly to various areas.

In recent years, some new fractional Sturm-Liouville problems have been
studied (see \cite{Bp}, \cite{al}, \cite{kro}, \cite{kli}, \cite{riv}).
These problems appear in various branches of natural sciences (see \cite{bal}%
, \cite{main}, \cite{mon}, \cite{pal}, \cite{sil}). Although the inverse
Sturm-Liouville problems with classical derivation are studied extensively,
there is only one study about this subject with conformable fractional
derivation. Mortazaasl and Akbarfam gave a solution of inverse nodal problem
for conformable fractional Sturm-Liouville operator in \cite{Oz}.

In the present paper, we consider a conformable fractional Sturm-Liouville
boundary value problem and give uniqueness theorems for the solution of
inverse problem according to the Weyl function, two eigenvalues-sets and the
sequences which consist of eigenvalues and norming constants. We also study
on half-inverse problem and prove a Hochstadt and Lieberman-type theorem.

\section{\textbf{Preliminaries}}

Before presenting our main results, we recall the some important concepts of
the conformable fractional calculus theory.

\begin{definition}
Let $f:[0,\infty )\rightarrow 
\mathbb{R}
$ be a given function. Then, the conformable fractional derivative of order $%
0<\alpha \leq 1$ of $f$ \ at $x>0$ is defined by:%
\begin{equation*}
D^{\alpha }f(x)=\underset{h\rightarrow 0}{\lim }\frac{f(x+hx^{1-\alpha
})-f(x)}{h},
\end{equation*}%
and the fractional derivative at $0$ is defined as $D^{\alpha }f(0)=\underset%
{x\rightarrow 0^{+}}{\lim }D^{\alpha }f(x).$
\end{definition}

\begin{definition}
Let $f:[0,\infty )\rightarrow 
\mathbb{R}
$ be a given function. The conformable fractional integral of $f$ of order $%
\alpha $ is defined by:%
\begin{equation*}
I_{\alpha }f(x)=\int\limits_{0}^{x}f(t)d_{\alpha
}t=\int\limits_{0}^{x}t^{\alpha -1}f(t)dt,
\end{equation*}%
for all $x>0.$
\end{definition}

We collect some necessary relations in the following lemma.

\begin{lemma}
\bigskip Let $f,g$ be $\alpha $-differentiable at $x,$ $x>0.$

i) $D_{x}^{\alpha }(af+bg)=aD_{x}^{\alpha }f+bD_{x}^{\alpha }g,$ $\forall
a,b\in 
\mathbb{R}
,$

ii) $D_{x}^{\alpha }(x^{a})=ax^{a-\alpha },$ $\forall a\in 
\mathbb{R}
,$

iii) $D_{x}^{\alpha }(c)=0,$ ($c$ is a constant)

iv) $D_{x}^{\alpha }(fg)=D_{x}^{\alpha }(f)g+fD_{x}^{\alpha }(g)$,\newline

v) $D_{x}^{\alpha }(f/g)=\frac{D_{x}^{\alpha }(f)g-fD_{x}^{\alpha }(g)}{g^{2}%
},$

vi) if $f$ is a continuous function, then for all $x>0,$ we have $\
D_{x}^{\alpha }I_{\alpha }f(x)=f(x),$

vii) if $f$ is a differentiable function, then we have$\ D_{x}^{\alpha
}f(x)=x^{1-\alpha }f^{^{\prime }}(x),$
\end{lemma}

\begin{lemma}
Let $f,g:(0,\infty )\rightarrow 
\mathbb{R}
$ be $\alpha $-differentiable functions and $h(x)=f(g(x)).$ Then, $h(x)$ is $%
\alpha $-differentiable and for all $x\neq 0$ and $g(x)\neq 0,$%
\begin{equation*}
(D_{x}^{\alpha }h)(x)=(D_{x}^{\alpha }f)(g(x))(D_{x}^{\alpha }g)(x)g^{\alpha
-1}(x),
\end{equation*}%
if $x=0,$ then $\ \ (D_{x}^{\alpha }h)(0)=\underset{x\rightarrow 0^{+}}{\lim 
}(D_{x}^{\alpha }f)(g(x))(D_{x}^{\alpha }g)(x)g^{\alpha -1}(x).$
\end{lemma}

\bigskip For further knowledge about the conformable fractional derivative,
the reader is referred to \cite{abd} and \cite{atan}, \cite{hal}.

Let us consider the following boundary value problem $L_{\alpha }(q(x),h,H)$%
\ 
\begin{eqnarray}
&&\text{\ }\left. \ell y:=-D_{x}^{\alpha }D_{x}^{\alpha }y+q(x)y=\lambda y%
\text{, \ \ }0<x<\pi \right. \medskip \\
&&\text{ }\left. U(y):=D_{x}^{\alpha }y(0)-hy(0)=0\right. \medskip \\
&&\text{ }\left. V(y):=D_{x}^{\alpha }y(\pi )+Hy(\pi )=0\right. \medskip
\end{eqnarray}%
where $D_{x}^{\alpha }$ is the conformable fractional (CF) derivative of
order $\alpha ,$ $0<\alpha \leq 1,$ $q(t)$ is real valued continuous
function on $\left[ 0,\pi \right] $, $h,H\in 
\mathbb{R}
$ and $\lambda $ is the spectral parameter.

Let the functions $\varphi (x,\lambda )$ and $\psi (x,\lambda )$ be the
solutions of (1) under the initial conditions 
\begin{equation}
\varphi (0,\lambda )=1\text{, }D_{x}^{\alpha }\varphi (0,\lambda )=h\text{
and }\psi (\pi ,\lambda )=1,D_{x}^{\alpha }\psi (\pi ,\lambda )=-H
\end{equation}%
respectively. These solutions are entire according to $\lambda $ for each
fixed $x$ in $\left[ 0,\pi \right] $ and they satisfy the following
asymptotic formulas \cite{Oz}: 
\begin{eqnarray}
\varphi (x,\lambda ) &=&\cos (\frac{\sqrt{\lambda }}{\alpha }x^{\alpha
})+O\left( \dfrac{1}{\sqrt{\lambda }}\exp (\frac{\left\vert \tau \right\vert 
}{\alpha }x^{\alpha })\right) \\
D_{x}^{\alpha }\varphi (x,\lambda ) &=&-\sqrt{\lambda }\sin (\frac{\sqrt{%
\lambda }}{\alpha }x^{\alpha })+O\left( \exp (\frac{\left\vert \tau
\right\vert }{\alpha }x^{\alpha })\right) \\
\psi (x,\lambda ) &=&\cos (\frac{\sqrt{\lambda }}{\alpha }(\pi ^{\alpha
}-x^{\alpha }))+O\left( \dfrac{1}{\sqrt{\lambda }}\exp (\frac{\left\vert
\tau \right\vert }{\alpha }(\pi ^{\alpha }-x^{\alpha }))\right) \\
D_{x}^{\alpha }\psi (x,\lambda ) &=&\sqrt{\lambda }\sin (\frac{\sqrt{\lambda 
}}{\alpha }(\pi ^{\alpha }-x^{\alpha }))+O\left( \exp (\frac{\left\vert \tau
\right\vert }{\alpha }(\pi ^{\alpha }-x^{\alpha }))\right)
\end{eqnarray}%
where $\tau :=\func{Im}\sqrt{\lambda }.$ The function

\begin{equation*}
W_{\alpha }[\psi (x,\lambda ),\varphi (x,\lambda )]=\psi (x,\lambda
)D_{x}^{\alpha }\varphi (x,\lambda )-\varphi (x,\lambda )D_{x}^{\alpha }\psi
(x,\lambda )
\end{equation*}%
is called as the fractional Wronskian of $\psi $ and $\varphi .$\ It is
proven in \cite{Oz} that $W_{\alpha }$ does not depend on $x$ and it can be
written as $W_{\alpha }[\psi (x,\lambda ),\varphi (x,\lambda )]=\Delta
(\lambda )=V(\varphi )=-U(\psi ).$\ Put $G_{\delta }:=\left\{ \sqrt{\lambda }%
:\left\vert \sqrt{\lambda }-\frac{\alpha }{\pi ^{\alpha -1}}k\right\vert
\geq \delta ,\text{ }k=0,1,2,\ldots \right\} ,$ where $\delta $ is a
sufficiently small positive number. The function $\Delta (\lambda )$
satisfies the inequality 
\begin{equation}
\left\vert \Delta (\lambda )\right\vert \geq C_{\delta }\left\vert \sqrt{%
\lambda }\right\vert \exp (\frac{\left\vert \tau \right\vert }{\alpha }\pi
^{\alpha }),\text{ \ }\sqrt{\lambda }\in G_{\delta },\left\vert \sqrt{%
\lambda }\right\vert \geq \rho ^{\ast },
\end{equation}%
for sufficiently large $\rho ^{\ast }=\rho ^{\ast }(\delta ).$

Let $\left\{ \lambda _{n}\right\} _{n\geq 0}$ be the eigenvalues sets of $%
L_{\alpha }(q(x),h,H)$. The numbers $\lambda _{n}$ are real, simple and
satisfy the following asymptotic estimate:%
\begin{equation}
\sqrt{\lambda _{n}}=\left( \frac{\alpha }{\pi ^{\alpha -1}}\right) n+\frac{%
\omega _{\alpha }}{n\pi }+\frac{\kappa _{\alpha _{n}}}{n},\text{ }\kappa
_{\alpha _{n}}\in l_{\alpha }^{2}
\end{equation}%
where $\omega _{\alpha }=h+H+\frac{1}{2}\int\limits_{0}^{\pi }q(t)d_{\alpha
}t$ \cite{Oz}.

\section{\textbf{Uniqueness Theorems}}

Together with $L_{\alpha }$, we consider a boundary value problem $%
\widetilde{L}_{\alpha }=L(\widetilde{q}(x),\widetilde{h},\widetilde{H})$ of
the same form but with different coefficients. We assume that if a certain
symbol $s$ denotes an object related to $L_{\alpha }$ , then $\widetilde{s}$
will denote an analogous object related to $\widetilde{L}_{\alpha }.$

\subsection{According to the Weyl function}

Let $S(x,\lambda )$\ be a solution of (1) that satisfies the conditions $%
S(0,\lambda )=1$, $D_{x}^{\alpha }S(0,\lambda )=1.$ It is clear that $%
W_{\alpha }[\varphi (x,\lambda ),S(x,\lambda )]=1$ and the function $\psi
(x,\lambda )$ can be represented by

\begin{equation}
\frac{\psi (x,\lambda )}{\Delta (\lambda )}=S(x,\lambda )-M(\lambda )\varphi
(x,\lambda )
\end{equation}%
where $M(\lambda )=\dfrac{-\psi (0,\lambda )}{\Delta (\lambda )}$ is called
as Weyl function.

\begin{theorem}
If $M(\lambda )=\widetilde{M}(\lambda )$ then $q(x)=\widetilde{q}(x),$ a.e.
in $\left[ 0,\pi \right] $, $h=\widetilde{h}$ and $H=\widetilde{H}.$
\end{theorem}

\begin{proof}
Let us consider the functions $P_{1}(x,\lambda )$ and $P_{2}(x,\lambda )$
which are defined by the following formulas%
\begin{eqnarray}
P_{1}(x,\lambda ) &=&\varphi (x,\lambda )D_{x}^{\alpha }\widetilde{\phi }%
(x,\lambda )-\phi (x,\lambda )D_{x}^{\alpha }\widetilde{\varphi }(x,\lambda
), \\
P_{2}(x,\lambda ) &=&\phi (x,\lambda )\widetilde{\varphi }(x,\lambda
)-\varphi (x,\lambda )\widetilde{\phi }(x,\lambda ),
\end{eqnarray}%
where $\phi (x,\lambda )=\frac{\psi (x,\lambda )}{\Delta (\lambda )}.$ It is
easy to see that the functions $P_{1}(x,\lambda )$ and $P_{2}(x,\lambda )$
are meromorphic with respect to $\lambda .$ Moreover, $M(\lambda )=%
\widetilde{M}(\lambda )$ and (11) yield that $P_{1}$ and $P_{2}$ are entire
in $\lambda .$ It follows from asymptotic formulas (5)-(9) that $%
P_{1}(x,\lambda )=O(1)$ and $P_{2}(x,\lambda )=O\left( \dfrac{1}{\sqrt{%
\lambda }}\right) $. Therefore, we obtain $P_{1}(x,\lambda )=A(x)$ and $%
P_{2}(x,\lambda )=0$ by well-known Liouville's theorem. From (12) and (13)
we get 
\begin{equation*}
\varphi (x,\lambda )=A(x)\widetilde{\varphi }(x,\lambda )\text{ and }\phi
(x,\lambda )=A(x)\widetilde{\phi }(x,\lambda ).
\end{equation*}%
On the other hand, since%
\begin{equation*}
W_{\alpha }[\varphi (x,\lambda ),\phi (x,\lambda )]=\frac{W_{\alpha
}[\varphi (x,\lambda ),\psi (x,\lambda )]}{\Delta (\lambda )}=1
\end{equation*}%
and similarly $W_{\alpha }[\widetilde{\varphi }(x,\lambda ),\widetilde{\phi }%
(x,\lambda )]=1,$ then\ $A^{2}(x)=1$ for all $x\in \left[ 0,\pi \right] .$
Taking into account the asymptotic expressions of $\varphi (x,\lambda )$ and 
$\widetilde{\varphi }(x,\lambda ),$ we get $A(x)\equiv 1.$ Hence $\varphi
(x,\lambda )=\widetilde{\varphi }(x,\lambda ),$ $\phi (x,\lambda )=%
\widetilde{\phi }(x,\lambda )$ and so $q(x)=\widetilde{q}(x),$ a.e. in $%
\left[ 0,\pi \right] $, $h=\widetilde{h}$ and $H=\widetilde{H}.$
\end{proof}

\subsection{According to two given spectra or a spectrum and norming
cons-tants}

We consider the boundary value problem $L_{\alpha ,1}$ with the condition $%
y(0,\lambda )=0$ instead of (2) in $L_{\alpha }$. Let $\{\xi _{n}\}_{n\geq
0} $ be the eigenvalues of the problem $L_{\alpha ,1}$. It is obvious that $%
\xi _{n}$ are zeros of $\Delta _{1}(\xi ):=\psi (0,\xi ).$

\begin{theorem}
If$\ \left\{ \lambda _{n},\xi _{n}\right\} _{n\geq 0}=\left\{ \widetilde{%
\lambda }_{n},\widetilde{\xi }_{n}\right\} _{n\geq 0}$ then $q(x)=\widetilde{%
q}(x),$ a.e. in $\left[ 0,\pi \right] $, $h=\widetilde{h}$ and $H=\widetilde{%
H}.$
\end{theorem}

\begin{proof}
According to the Theorem 3.11 in \cite{Oz}, the function $\Delta (\lambda )$
can be represented as follows%
\begin{equation}
\Delta (\lambda )=\frac{\pi ^{3\alpha -2}}{\alpha ^{3}}\left( \lambda
_{0}-\lambda \right) \prod\limits_{n=0}^{\infty }\left( \dfrac{\lambda
_{n}-\lambda }{n^{2}}\right) .
\end{equation}%
Therefore, $\Delta (\lambda )\equiv \widetilde{\Delta }(\lambda )$
(similarly $\Delta _{1}(\xi )\equiv \widetilde{\Delta }_{1}(\xi )$)$,$ when $%
\lambda _{n}=\widetilde{\lambda }_{n}$ ($\xi _{n}=\widetilde{\xi }_{n}$) for
all $n.$ Consequently, $M(\lambda )\equiv \widetilde{M}(\lambda )\ $and so
the proof is completed by Theorem 1.
\end{proof}

\begin{lemma}[\protect\cite{Oz}, Lemma 3.5]
Denote $\alpha _{n}=\left\Vert \varphi (x,\lambda _{n})\right\Vert
_{2,\alpha }^{2}=\int\limits_{0}^{\pi }\varphi ^{2}(x,\lambda _{n})d_{\alpha
}x.$ Then, we have $\beta _{n}\alpha _{n}=-\Delta ^{\prime }(\lambda _{n})$,
where $\beta _{n}=\varphi (\pi ,\lambda _{n})=\dfrac{1}{\psi (0,\lambda _{n})%
}$.
\end{lemma}

The numbers $\alpha _{n}$ in Lemma 1 are called norming constants.

\begin{theorem}
If $\left\{ \lambda _{n},\alpha _{n}\right\} _{n\geq 0}=\left\{ \widetilde{%
\lambda }_{n},\widetilde{\alpha }_{n}\right\} _{n\geq 0}$ then $q(x)=%
\widetilde{q}(x),$ a.e. in $\left[ 0,\pi \right] $, $h=\widetilde{h}$ and $H=%
\widetilde{H}$.
\end{theorem}

\begin{proof}
Since $\lambda _{n}=\widetilde{\lambda }_{n},$ $\Delta (\lambda )\equiv 
\widetilde{\Delta }(\lambda ).$ Therefore, it is obtained by using Lemma 3
that $\beta _{n}=\widetilde{\beta }_{n}$ and so $\psi (0,\lambda _{n})=%
\widetilde{\psi }(0,\lambda _{n}).$ Hence the function 
\begin{equation*}
G(\lambda ):=\dfrac{\psi (0,\lambda )-\widetilde{\psi }(0,\lambda )}{\Delta
(\lambda )}
\end{equation*}%
is entire on $\lambda .$ Moreover, one can obtained from (7) and (9) that $%
G(\lambda )=O(\dfrac{1}{\lambda })$ for $\left\vert \lambda \right\vert
\rightarrow \infty .$ Thus $G(\lambda )\equiv 0$ and $\psi (0,\lambda
)\equiv \widetilde{\psi }(0,\lambda ).$ Finally, we get $M(\lambda )=%
\widetilde{M}(\lambda )$ and so obtain our desired result by Theorem 1.
\end{proof}

\subsection{According to the mixed data}

The next theorem is a generalized version of well-known Hochstadt and
Lieberman theorem in the classical Sturm-Liouville theory.

\begin{theorem}
If $\left\{ \lambda _{n}\right\} _{n\geq 0}=\left\{ \widetilde{\lambda }%
_{n}\right\} _{n\geq 0}$, $H=\widetilde{H}$ and $q(x)=\widetilde{q}(x)$ on $%
\left( \pi /2,\pi \right) $ then $q(x)=\widetilde{q}(x),$ a.e. in $[0,\pi ]$
and $h=\widetilde{h}$.
\end{theorem}

\begin{proof}
It is clear that the following equality holds%
\begin{equation}
D_{x}^{\alpha }\left[ \widetilde{\varphi }(x,\lambda )D_{x}^{\alpha }\varphi
(x,\lambda )-\varphi (x,\lambda )D_{x}^{\alpha }\widetilde{\varphi }%
(x,\lambda )\right] =\left[ q(x)-\widetilde{q}(x)\right] \varphi (x,\lambda )%
\widetilde{\varphi }(x,\lambda )
\end{equation}%
By integrating (in the conformable fractional integral) both sides of this
equality on $\left[ 0,\pi \right] ,$ we obtain%
\begin{equation*}
\left[ \widetilde{\varphi }(x,\lambda )D_{x}^{\alpha }\varphi (x,\lambda
)-\varphi (x,\lambda )D_{x}^{\alpha }\widetilde{\varphi }(x,\lambda )\right]
_{0}^{\pi }=\int\limits_{0}^{\pi }\left[ q(t)-\widetilde{q}(t)\right]
\varphi (t,\lambda )\widetilde{\varphi }(t,\lambda )d_{\alpha }t.
\end{equation*}%
Since $q(x)=\widetilde{q}(x)$ on $\left( \pi /2,\pi \right) $ and from (4),
it is obvious that 
\begin{equation*}
\left. \widetilde{\varphi }(\pi ,\lambda )D_{x}^{\alpha }\varphi (\pi
,\lambda )-\varphi (\pi ,\lambda )D_{x}^{\alpha }\widetilde{\varphi }(\pi
,\lambda )=h-\widetilde{h}+\int\limits_{0}^{\pi /2}\left[ q(t)-\widetilde{q}%
(t)\right] \varphi (t,\lambda )\widetilde{\varphi }(t,\lambda )d_{\alpha
}t\right.
\end{equation*}%
Let%
\begin{equation*}
\left. H(\lambda ):=h-\widetilde{h}+\int\limits_{0}^{\pi /2}\left[ q(t)-%
\widetilde{q}(t)\right] \varphi (t,\lambda )\widetilde{\varphi }(t,\lambda
)d_{\alpha }t\right.
\end{equation*}%
Since $\widetilde{\varphi }(\pi ,\lambda _{n})D_{x}^{\alpha }\varphi (\pi
,\lambda _{n})-\varphi (\pi ,\lambda _{n})D_{x}^{\alpha }\widetilde{\varphi }%
(\pi ,\lambda _{n})=0,$ $H(\lambda _{n})=0$ for all $n$ and so $\chi
(\lambda ):=\dfrac{H(\lambda )}{\Delta (\lambda )}$ is entire on $\lambda .$
On the other hand, from the asymptotic expressions of $\varphi (x,\lambda )$
and $\widetilde{\varphi }(x,\lambda ),$ it can be calculated that $%
\left\vert \chi (\lambda )\right\vert \leq \frac{C}{\left\vert \sqrt{\lambda 
}\right\vert }$ for sufficiently large $\left\vert \lambda \right\vert $. By
Liouville's Theorem, we get $\chi (\lambda )=0$ for all $\lambda $. Hence $%
H(\lambda )\equiv 0$.

By integrating\ again both sides of the equality (15) on $\left( 0,\pi
/2\right) $, we get%
\begin{equation}
\widetilde{\varphi }(\pi /2,\lambda )D_{x}^{\alpha }\varphi (\pi /2,\lambda
)=\varphi (\pi /2,\lambda )D_{x}^{\alpha }\widetilde{\varphi }(\pi
/2,\lambda )
\end{equation}%
Put $\psi (x,\lambda ):=\varphi (\left( \left( \pi /2\right) ^{\alpha
}-x^{\alpha }\right) ^{1/\alpha },\lambda ).$ From Lemma 2, it is clear that 
$\psi (x,\lambda )$ is the solution of the following initial value problem 
\begin{eqnarray*}
&&\left. -D_{x}^{\alpha }D_{x}^{\alpha }y+q(\left( \left( \pi /2\right)
^{\alpha }-x^{\alpha }\right) ^{1/\alpha })y=\lambda y,\text{ }x\in (0,\pi
/2)\right. \\
&&\left. y(\pi /2,\lambda )=1,\text{ }D_{x}^{\alpha }y(\pi /2,\lambda
)=-h\right.
\end{eqnarray*}%
It follows from (16) that%
\begin{equation}
\widetilde{\psi }(0,\lambda )D_{x}^{\alpha }\psi (0,\lambda )=\psi
(0,\lambda )D_{x}^{\alpha }\widetilde{\psi }(0,\lambda ).
\end{equation}%
Taking into account Theorem 1, it is concluded that $q(x)=$ $\widetilde{q}%
(x) $ on $\left[ 0,\pi /2\right] $ and $h=\widetilde{h}.$\ This completes
the proof.
\end{proof}

\end{document}